\documentclass{amsart}
\usepackage{graphicx}
\usepackage{amssymb}
\usepackage{epstopdf}
\usepackage{amsmath,amsthm,amscd,amssymb}
\usepackage{latexsym}
\usepackage[colorlinks,citecolor=red,pagebackref,hypertexnames=false]{hyperref}
\usepackage{geometry}                
\usepackage{color, xcolor}
\numberwithin{equation}{section}

\theoremstyle{plain}
\newtheorem{theorem}{Theorem}[section]
\newtheorem{lemma}[theorem]{Lemma}

\theoremstyle{definition}
\newtheorem{definition}[theorem]{Definition}

\newtheorem{case[theorem]}{Case}

\theoremstyle{remark}
\newtheorem{remark}[theorem]{Remark}

\numberwithin{equation}{section}

\begin{document}

\title{Dot products in ${\Bbb F}_q^3$ and the Vapnik-Chervonenkis dimension}


\author{A. Iosevich, B. McDonald, and M. Sun}

\date{\today}

\address{Department of Mathematics, University of Rochester, Rochester, NY}
\email{iosevich@gmail.com}

\address{Department of Mathematics, University of Rochester, Rochester, NY}
\email{bmcdonald3879@gmail.com }

\address{Harley School}
\email{maxwell.sun.2022@harleystudents.org}

\thanks{The first listed author's research was partially supported by the NSF HDR Tripods 19343963. This research was initiated during the Tripods/StemForAll 2021 REU.}

\maketitle

\begin{abstract} Given a set $E \subset {\Bbb F}_q^3$, where ${\Bbb F}_q$ is the field with $q$ elements. Consider a set of "classifiers" ${\mathcal H}^3_t(E)=\{h_y: y \in E\}$, where $h_y(x)=1$ if $x \cdot y=t$, $x \in E$, and $0$ otherwise. We are going to prove that if $|E| \ge Cq^{\frac{11}{4}}$, with a sufficiently large constant $C>0$, then the Vapnik-Chervonenkis dimension of ${\mathcal H}^3_t(E)$ is equal to $3$. In particular, this means that for sufficiently large subsets of ${\Bbb F}_q^3$, the Vapnik-Chervonenkis dimension of ${\mathcal H}^3_t(E)$ is the same as the Vapnik-Chervonenkis dimension of ${\mathcal H}^3_t({\Bbb F}_q^3)$. In some sense the proof leads us to consider the most complicated possible configuration that can always be embedded in subsets of ${\Bbb F}_q^3$ of size $\ge Cq^{\frac{11}{4}}$. 

\vskip.125in 

This paper is dedicated to the Ukrainian people who are suffering the effects of a brutal aggression. 

\end{abstract} 


\section{Introduction} 

\vskip.125in 

The purpose of this paper is to study the Vapnik-Chervonenkis dimension in the context of a naturally arising family of functions on subsets of the three-dimensional vector space over the finite field with $q$ elements, denoted by ${\Bbb F}_q^3$. Let us begin by recalling some definitions and basic results (see e.g. \cite{DS14}, Chapter 6). 

\begin{definition} \label{shatteringdef} Let $X$ be a set and ${\mathcal H}$ a collection of functions from $X$ to $\{0,1\}$. We say that ${\mathcal H}$ shatters a finite set $C \subset X$ if the restriction of ${\mathcal H}$ to $C$ yields every possible function from $C$ to $\{0,1\}$. \end{definition} 

\vskip.125in 

\begin{definition} \label{vcdimdef} Let $X$ and ${\mathcal H}$ be as above. We say that a non-negative integer $n$ is the VC-dimension of ${\mathcal H}$ if there exists a set $C \subset X$ of size $n$ that is shattered by ${\mathcal H}$, and no subset of $X$ of size $n+1$ is shattered by ${\mathcal H}$. \end{definition} 

\vskip.125in 

We are going to work with a class of functions ${\mathcal H}^d_t$, where $t \not=0$. Let $X={\Bbb F}_q^d$, and define 
\begin{equation} \label{functionclassdef} {\mathcal H}_t^d=\{h_y: y \in {\Bbb F}_q^d \}, \end{equation} where $y \in {\Bbb F}_q^d$, and $h_y(x)=1$ if $x \cdot y=t$, and $0$ otherwise. Let ${\mathcal H}_t^2(E)$ be defined the same way, but with respect to a set $E \subset {\Bbb F}_q^d$ i.e 
$$ {\mathcal H}^d_t(E)=\{h_y: y \in E\},$$ where $h_y(x)=1$ if $x \cdot y=t$ ($x \in E$), and $0$ otherwise. 

\vskip.125in 

Our main result is the following. 

\begin{theorem} \label{main} Let ${\mathcal H}^3_t(E)$ be defined as above with respect to $E \subset {\Bbb F}_q^3$, $t \not=0$. If 
$|E| \ge Cq^{\frac{11}{4}}$, for some large enough constant $C$, then the VC-dimension of ${\mathcal H}^3_t(E)$ is equal to $3$. \end{theorem} 

\vskip.125in 

\begin{remark} Since $|{\mathcal H}_t^3(E)|=|E|$, it is clear that the VC-dimension of ${\mathcal H}_t^3(E)$ is at most $\log_2(|E|)$, so $3$ is a clear improvement over this general estimate. It is not difficult to see that the VC-dimension is $<4$ since three points determine a plane in ${\Bbb F}_q^3$, so the real challenge is to establish that some set of 3 points shatters. Moreover, our result says that in this sense, the learning complexity of subsets of ${\Bbb F}_q^3$ of size $>Cq^{\frac{11}{4}}$ is the same as that of the whole vector space ${\Bbb F}_q^3$. \end{remark} 

\vskip.125in 

\begin{remark} In the case when $d=2$ and the dot product $x \cdot y$ is replaced by $||x-y||={(x_1-y_1)}^2+{(x_2-y_2)}^2$, the corresponding result, with the threshold $|E| \ge Cq^{\frac{15}{8}}$ was established by D. Fitzpatrick, E. Wyman and the first two listed authors of this paper (\cite{FIMW21}). The techniques used to prove Theorem \ref{main} are quite a bit different. On one hand, we have more room to roam in three dimensions. On the other, the non-translation invariant nature of the dot product requires special care. \end{remark} 

\vskip.125in 

\begin{remark} As the reader shall see, the proof of Theorem \ref{main} involves a construction of a reasonably complicated point configuration in $E$. For a general theory of such configurations in the context of dot products, see e.g. \cite{HIKR11} and \cite{PV17}. \end{remark} 

\vskip.125in 

\begin{remark} The concept of the VC-dimension plays an important role in many combinatorial problems. See, for example, \cite{AS16}, \cite{HW87}, and the references contained therein. \end{remark} 

\vskip.125in 

We can also prove that the VC-dimension is $\ge 2$ under a much weaker assumption. More precisely, we have the following result. 

\begin{theorem} \label{mainsecond} Let ${\mathcal H}^3_t(E)$ be defined as above with respect to $E \subset {\Bbb F}_q^3$, $t \not=0$. If $|E|>cq^{\frac{5}{2}}$ for an arbitrary $c$, then the VC-dimension of ${\mathcal H}^3_t(E)$ is $\ge 2$. \end{theorem} 

\begin{remark} We do not know to what extent the exponent $\frac{11}{4}$ in Theorem \ref{main} and the exponent $\frac{5}{2}$ are sharp, but we know that neither exponent can fall below $2$. \end{remark} 

\vskip.125in 

\section{Learning theory perspective on Theorem \ref{main}} 

\vskip.125in 

From the point of view of learning theory, it is interesting to ask what the "learning task" is in the situation at hand. It can be described as follows. We are asked to construct a function $f:E \to \{0,1\}$, $E \subset {\Bbb F}_q^3$, that is equal to $1$ when $x \cdot y^{*}=t$, but we do not know the value of $y^{*}$. The fundamental theorem of statistical learning tells us that if the VC-dimension of ${\mathcal H}_t^3(E)$ is finite, we can find an arbitrarily accurate hypothesis (element of ${\mathcal H}^3_t(E)$ with arbitrarily high probability if we consider a randomly chosen sampling training set of sufficiently large size. 

We shall now make these concepts precise. Let us recall some more basic notions. 

\begin{definition} Given a set $X$, a probability distribution $D$ and a labeling function $f: X \to \{0,1\}$, let $h$ be a hypothesis, i.e $h:X \to \{0,1\}$, and define 
$$ L_{D,f}(h)={\Bbb P}_{x \sim D}[h(x) \not=f(x)],$$ where ${\Bbb P}_{x \sim D}$ means that $x$ is being sampled according to the probability distribution $D$. 
\end{definition} 

\vskip.125in 

\begin{definition} A hypothesis class ${\mathcal H}$ is PAC learnable if there exist a function 
$$m_{{\mathcal H}}: {(0,1)}^2 \to {\Bbb N}$$ and a learning algorithm with the following property: For every $\epsilon, \delta \in (0,1)$, for every distribution $D$ over $X$, and for every labeling function $f: X \to \{0,1\},$ if the realizability assumption holds with respect to $X$, $D$, $f$, then when running the learning algorithm on $m \ge m_{{\mathcal H}}(\epsilon, \delta)$ i.i.d. examples generated by $D$, and labeled by $f$, the algorithm returns
a hypothesis $h$ such that, with probability of at least $1-\delta$, (over the choice of the examples), 
$$L_{D,f}(h) \leq \epsilon.$$
\end{definition} 

\vskip.125in 

\begin{theorem} Let ${\mathcal H}$ be a collection of hypotheses on a set $X$. Then ${\mathcal H}$ has a finite VC-dimension if and only if ${\mathcal H}$ is PAC learnable. Moreover, if the VC-dimension of ${\mathcal H}$ is equal to $n$, then ${\mathcal H}$ is PAC learnable and there exist constants $C_1, C_2$ such that 
$$ C_1 \frac{n+\log \left(\frac{1}{\delta} \right)}{\epsilon} \leq m_{{\mathcal H}}(\epsilon, \delta) \leq C_2 \frac{n \log \left(\frac{1}{\epsilon} \right)+
\log \left(\frac{1}{\delta} \right)}{\epsilon}.$$

\end{theorem} 

\vskip.125in 

Going back to the learning task associated with ${\mathcal H}_t^3(E)$, as in Theorem \ref{main}, suppose that $h_y$ is a "wrong" hypothesis, i.e 
$y \not=y^{*}$, where $f=h_{y^{*}}$ is the true labeling function. 

Since the size of a plane in ${\Bbb F}_q^3$ is $q^2$, and $D$ is the uniform probability distribution on 
${\Bbb F}_q^3$, 
$$ L_{D,f}(h) \leq \frac{1}{q} \left(1+o(1) \right),$$ so one must choose $\epsilon$ just slightly less than $\frac{1}{q}$ to make the results meaningful. It follows by taking 
$\delta=\epsilon$ that we need to consider random samples of size $\approx C q \log(q)$ with sufficiently large $C$ to execute the desired algorithm. Moreover, since $3$ points determine a plane in ${\Bbb F}_q^3$ effectively means that if $\epsilon$ is just slightly less than $\frac{1}{q}$, then $L_{D,f}(h)=0$. 

\vskip.25in 

\section{Proof of Theorem \ref{mainsecond}} 

\vskip.125in 

We prove Theorem \ref{mainsecond} first because some of the ideas in the proof will be needed in the proof of Theorem \ref{main}. It is sufficient to prove that there exist $x_1,x_2, y_1, y_2, y_{12}, y^{*} \in E$ such that \begin{itemize} 
\item i) $x_1 \cdot y_{12}=x_2 \cdot y_{12}=t$, \item ii) $x_1 \cdot y_1=t, x_2 \cdot y_1 \not=t$, \item iii) $x_2 \cdot y_2=t, x_1 \cdot y_2 \not=t$, \item iv) $x_1 \cdot y^{*}, x_2 \cdot y^{*} \not=t$ \end{itemize} 

It suffices to find such a tuple $(x_1,x_2,y_{12},y_1,y_2)$ under the additional assumption that for each $u\in E$, there are at most $C\frac{|E|}{q}$ vectors $v\in E$ such that $u\cdot v = t$.  The following lemma allows us to reduce to this case.

\begin{lemma}\label{prune}
Let $E\subseteq \mathbb{F}_q^3$ be a set satisfying the hypotheses of Theorem \ref{main}.  Then there is a subset $E'\subseteq E$ with $|E'|\geq \frac{1}{2}|E|$, and for any $u\in E'$, 
$$
\sum_{v\in E'}{D_t(u,v)}\leq \frac{22|E'|}{5q}
$$
\end{lemma}
Clearly if $\mathcal{H}_t^2(E')$ shatters some set of 3 points, then $\mathcal{H}_t^2(E)$ shatters the same set of 3 points.  Moreover, if $E$ satisfies the hypotheses of Theorem \ref{main} or Theorem \ref{mainsecond}, then so does $E'$.

\begin{proof}

It follows immediately from Theorem 2.3 in \cite{IJM2021} that
$$\sum_{u,v \in E}{D_t(u,v)} = \frac{|E|^2}{q} + R(t),$$
where $|R(t)| \le |E|q < \frac{|E|^2}{10q}$. In particular,
$$\sum_{u \in E}\sum_{v \in E}{D_t(u,v)} \le \frac{11|E|^2}{10q}.$$
This implies that at most $\frac{|E|}{2}$ distinct points $u\in E$ satisfy
$$
\sum_{v \in E}{D_t(u,v)} \ge \frac{11|E|}{5q}.
$$
Thus, for
$$
E':=\left\{u\in E: \sum_{v\in E}{D_t(u,v)}\leq \frac{11|E|}{5q}\right\},
$$
we see that $E'$ satisfies the conditions of the lemma.
\end{proof}
With this lemma, we may assume without loss of generality that there are at most $C\frac{|E|}{q}$ vectors $v\in E$ with $u\cdot v=t$.  By Theorem 2.2 in \cite{IJM2021}, there exist a set $P_5$ of ordered quintuples $(x_1,x_2,y_{12}, y_1,y_2)$ such that 
$$y_1 \cdot x_1=x_1 \cdot y_{12}=y_{12} \cdot x_2=x_2 \cdot y_2=t,$$ and 
$$\left||P_5| - \frac{|E|^5}{q^4}\right| \le \frac{4}{\log 2}q^2 \frac{|E|^4}{q^4} \le \frac{1}{2}\frac{|E|^5}{q^4}.$$

In particular, $|P_5| \ge \frac{1}{2}\frac{|E|^5}{q^4}$. Such a quintuple is represented in Figure \ref{mainsecond_mainfig} as a graph with the vectors as vertices and edges between them if their dot product is $t$.

\begin{figure}[h!]
	{\includegraphics[scale=0.15]{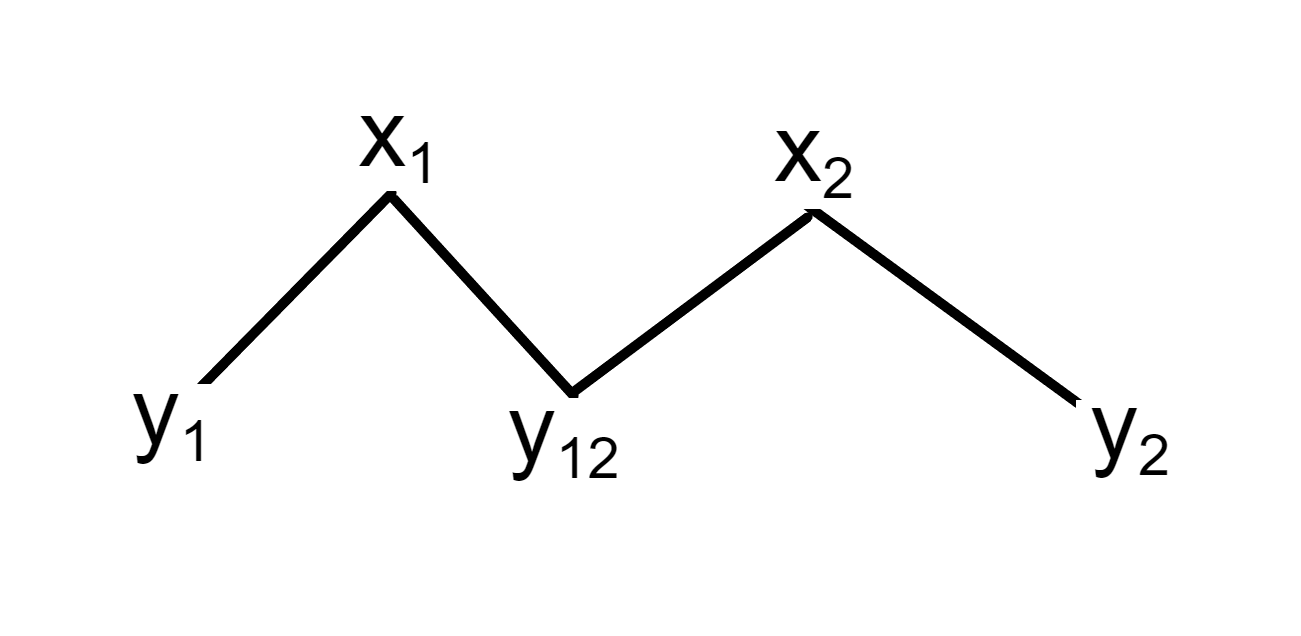} \caption{Configuration for Theorem \ref{mainsecond}} \label{mainsecond_mainfig}} 
\end{figure}

It remains to show that such a quintuple exists with $x_1 \cdot y_2 \not=t$ and $x_2 \cdot y_1 \not=t$. We first count the number of quintuples $(x_1,x_2,y_{12}, y_1,y_2)$ in $P_5$ with $x_1 \cdot y_2 =t$. This case of degeneracy is displayed in Figure \ref{mainsecond_degen_case} below. 

\begin{figure}[h!]
	\includegraphics[scale=0.3]{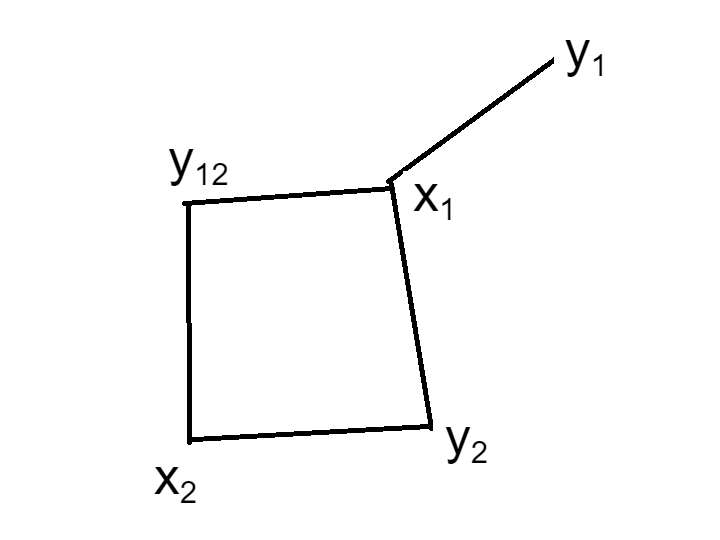}
	\caption{Degeneracy case in which $x_2 \cdot y_1 = t$}
	\label{mainsecond_degen_case}
\end{figure}

We have, 
\begin{align*}
	\sum_{x_1,x_2,y_{12},y_1,y_2 \in E}&{ D_t(y_1,x_1)D_t(x_1,y_{12})D_t(y_{12},x_2)D_t(x_2,y_2)D_t(y_2,x_1)} \\
	& = \sum_{x_1,x_2,y_{12},y_2 \in E}		{D_t(x_1,y_{12})D_t(y_{12},x_2)D_t(x_2,y_2)D_t(y_2,x_1)} \sum_{y_1 \in E}{D_t(y_1,x_1)} \\
	& \le \frac{22|E|}{5q} \sum_{x_1,x_2,y_{12},y_2 \in E}		{D_t(y_1,x_1)D_t(x_1,y_{12})D_t(y_{12},x_2)D_t(x_2,y_2)}	
\end{align*}
This sum over $x_1,x_2,y_{12},y_2$ is the number of 4-cycles in the dot-product graph on $E$, denoted $C_4^{prod}$ in the notation of \cite{IJM2021}. By Theorem 1.2 in \cite{IJM2021}, 
$$
	\left| \sum_{x_1,x_2,y_{12},y_2 \in E}{D_t(y_1,x_1)D_t(x_1,y_{12})D_t(y_{12},x_2)D_t(x_2,y_2)} - \frac{|E|^4}{q^4} \right| 
$$
$$
	\le \frac{|E|^4}{q^4} \left( 12q^{-\frac{1}{2}} + 8\frac{q^5}{|E|^2} + 28\frac{q^2}{|E|} \right) 
$$
$$
	\le \frac{|E|^4}{q^4} \left( 12q^{-\frac{1}{2}} + \frac{8}{c^2} + \frac{28}{c}q^{-\frac{1}{2}} \right)
$$
$$
	\le \frac{9|E|^4}{c^2q^4}
$$
Thus, 
$$\sum_{x_1,x_2,y_{12},y_2 \in E}{D_t(y_1,x_1)D_t(x_1,y_{12})D_t(y_{12},x_2)D_t(x_2,y_2)} \le \left( \frac{9}{c^2}+1 \right) 
\frac{|E|^4}{q^4}$$
and
\begin{align*}
	\sum_{x_1,x_2,y_{12},y_1,y_2 \in E}&{ D_t(y_1,x_1)D_t(x_1,y_{12})D_t(y_{12},x_2)D_t(x_2,y_2)D_t(y_2,x_1)} \\
	&\le \frac{22|E|}{5q} \left( \frac{9}{c^2}+1 \right) \frac{|E|^4}{q^4} \\
	&= \frac{22}{5}\left( \frac{9}{c^2}+1 \right) \frac{|E|^5}{q^5} < \frac{|E|^5}{10q^4}.
\end{align*}
That is, the number of quintuples $(x_1,x_2,y_{12}, y_1,y_2)$ in $P_5$ with $x_1 \cdot y_2 =t$ is less than $\frac{|E|^5}{10q^4}$. Analogously, the number of quintuples $(x_1,x_2,y_{12}, y_1,y_2)$ in $P_5$ with $x_2 \cdot y_1 =t$ is less than $\frac{|E|^5}{10q^4}$. It follows that there exists a quintuple $(x_1,x_2,y_{12}, y_1,y_2)$ in $P_5$ with $x_1 \cdot y_2, x_2 \cdot, y_1 \neq t$.

It only remains to construct $y^{*} \in E$ such that $x_1 \cdot y^{*} \not=t$ and $x_2 \cdot y^{*} \not=t$. Observe that 
$$ |\{x \in E: x \cdot x_1=t\}|=q^2,$$ and 
$$ |\{x \in E: x \cdot x_2=t\}|=q^2,$$ so since $|E|>2q^2$, there exists $y^{*}$ with the desired properties. This completes the proof of Theorem \ref{mainsecond}. 

\section{Proof of Theorem \ref{main}} 

As in the previous section, we may assume without loss of generality that for all $u \in E$,
$$
\sum_{v \in E}{D_t(u,v)} \le \frac{22|E|}{5q}.
$$
This time we will reduce to the case where the sum is bounded below as well, which follows analogously via a counterpart to Lemma \ref{prune}.
\begin{lemma}\label{prune2}
For a set $E$ satisfying the hypotheses of Theorem \ref{main}, there is a subset $E_0\subseteq E$ with $|E_0|\geq \frac{1}{6}|E|$, and for any $u\in E_0$, 
$$
\sum_{v\in E}{D_t(u,v)}\geq \frac{|E|}{5q}
$$
\end{lemma}
\begin{proof}
Let 
$$
E_0:=\left\{u\in E: \sum_{v\in E}{D_t(u,v)}\geq \frac{|E|}{5q}\right\},
$$
so that we need only show that $|E_0|\geq \frac{1}{6}|E|$.
\begin{align*}
	\sum_{u,v \in E}{D_t(u,v)} &= \sum_{u \in E_0}\sum_{v \in E}{D_t(u,v)} + \sum_{u \not\in E_0}\sum_{v \in E}{D_t(u,v)} \\
	&\le |E_0|\frac{22|E|}{5q} + (|E|-|E_0|)\frac{|E|}{5q} \\
	&= \frac{|E|^2}{5q} + |E_0|\frac{21|E|}{5q}.
\end{align*}
We know from the previous section that for $E$ satisfying the hypotheses of Theorem \ref{main},
$$
\left|\sum_{u,v\in E}{D_t(u,v)}-\frac{|E|^2}{q}\right|<\frac{|E|^2}{10q}.
$$
Thus,
$$\frac{9|E|^2}{10q} \le \sum_{u,v \in E}{D_t(u,v)} \le \frac{|E|^2}{5q} + |E_0|\frac{21|E|}{5q},$$
and so
$$|E_0| \ge \frac{|E|}{6}.$$
\end{proof}
With Lemma \ref{prune} and Lemma \ref{prune2}, we may assume that for all $u\in E$, 
$$
\frac{|E|}{5q}\leq \sum_{v\in E}{D_t(u,v)}\leq \frac{22|E|}{5q}
$$
In order to conclude that $\mathcal{H}_t^3(E)$ has VC-dimension 3, we need to find a set $\{x_1,x_2,x_3\}$ of 3 distinct points which is shattered by $\mathcal{H}_t^3(E)$.  This is equivalent to finding $x_1, x_2, x_3, y_1, y_2, y_3, y_{12}, y_{13}, y_{23}, y_{123}, y^* \in E$ with $x_1, x_2, x_3$ distinct such that
\begin{itemize}
\item $x_1 \cdot y_{123} = x_2 \cdot y_{123} = x_3 \cdot y_{123} = t$\
\item $x_1 \cdot y_{12} = x_2 \cdot y_{12} = t$, $x_3 \cdot y_{12} \neq t$, and similarly for $y_{13}$ and $y_{23}$
\item $x_1 \cdot y_1 = t$, $x_2 \cdot y_1, x_3 \cdot y_1 \neq t$, and similarly for $y_1$ and $y_2$
\item $x_1 \cdot y^*, x_2 \cdot y^*, x_3 \cdot y^* \neq t$
\end{itemize}
This configuration is displayed below in Figure \ref{main_whole_config}.

\begin{figure}[h!]
	\includegraphics[scale=0.23]{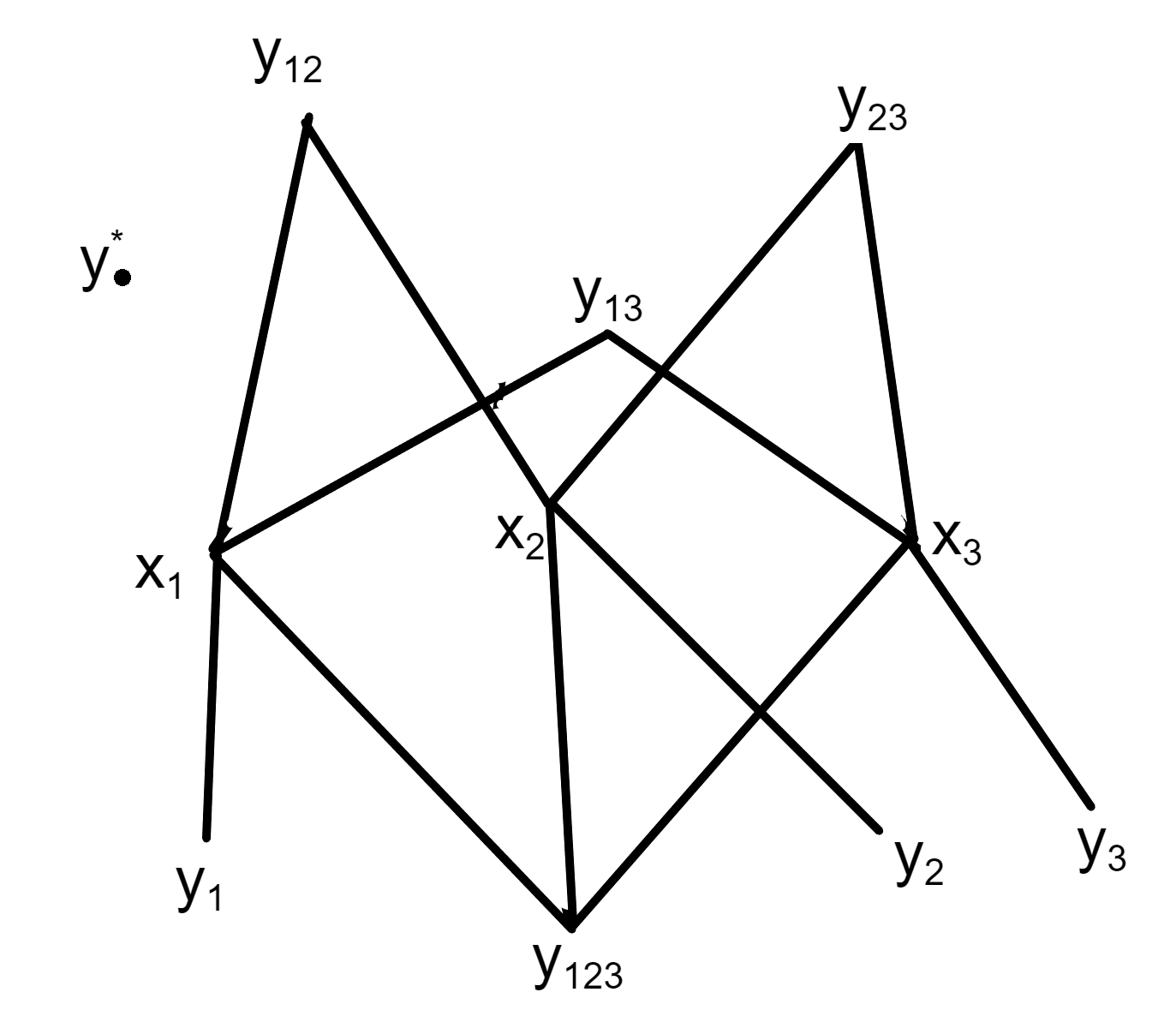}
	\caption{Configuration for shattering a set of three points}
	\label{main_whole_config}
\end{figure}
Let 
\begin{align}\label{setA}
A=\left\{(x,y,z,u,v)\in E^5: x\cdot y = y\cdot z = z\cdot u = u\cdot x = v\cdot u=t\right\}
\end{align}
\vspace{0.1 cm}

For $(x,y,z,u,v)\in A$, by identifying $x=y_{12}$, $y=x_2$, $z=y_{123}$, $u=x_1$, and $v=y_{13}$, this configuration corresponds to the graph shown in Figure \ref{main_halfconfig} below, which is a subgraph of the graph shown in Figure \ref{main_whole_config}.  Our strategy is to use the symmetry of the larger configuration, in the sense that by removing $y_1,y_2,y_3$ and then identifying $x_1=x_3$ and $y_{12}=y_{23}$, we obtain the smaller configuration.  

\begin{figure}[h!]
	\includegraphics[scale=0.2]{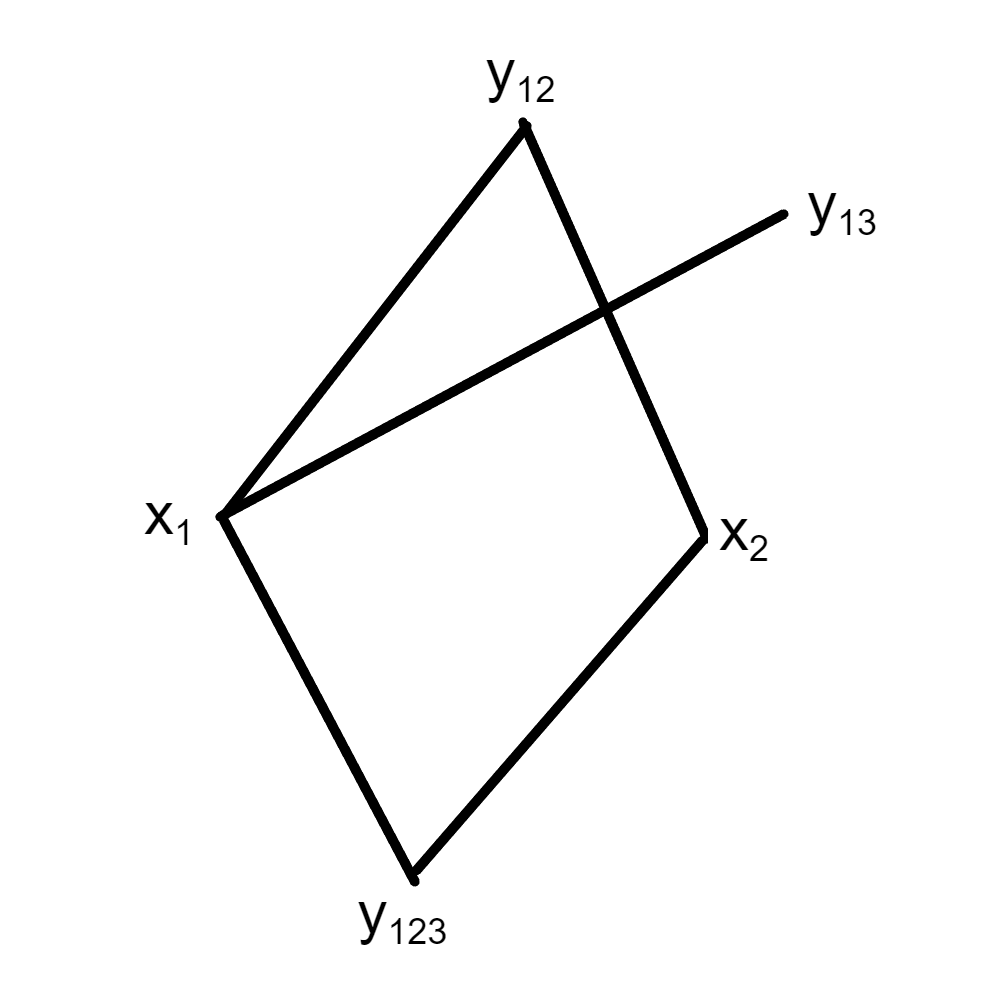}
	\caption{Initial configuration used to build up to full shattering configuration}
	\label{main_halfconfig}
\end{figure}

\vskip.125in 

We need the following result follows from \cite{J2022}, Chapter 2). 
\begin{lemma} Let $E \subset {\Bbb F}_q^3$ with $|E| \ge Cq^{\frac{5}{2}}$, $C$ sufficiently large. Then for $A$ as in equation \ref{setA},
$$ \frac{1}{2} {|E|}^5 q^{-5} \leq |A| \leq 2 {|E|}^5 q^{-5}.$$
\end{lemma} 
Since the definition of the set $A$ did not require that $x\neq z$, $y\neq u$, and $y\cdot v\neq t$, all of which will be necessary for our construction, we will find an upper bound for the number of elements of $A$ which do not have these properties.  The following lemma implies that these make up a small proportion of $A$.

\begin{lemma} \label{degeneracy_case}
Let $E \subset {\Bbb F}_q^3$ with $|E| \ge Cq^{\frac{5}{2}}$, $C$ sufficiently large. Then 
$$|\{(x,y,z,v,u) \in A: y\cdot v = t, \ \text{or} \ x=z, \ \text{or} \ y=u\}| \leq 5|E|^2q^3.$$
\end{lemma}
\begin{proof}
There are at most $|E|^2$ ways to produce a pair of distinct points $u,y$ in $E$. Then, the intersection of the planes defined by $a \cdot u = t$ and $a \cdot y = t$ is at most a line since these planes are distinct. There are at most $q^3$ ways to choose 3 points on that line, $x,z,u\in E$.  So, there are at most $|E|^2q^3$ quintuples $(x,y,z,v,u)\in A$ with $y\cdot v=t$. For the case when $y=u$, by Corollary 4.5 in \cite{IJM2021} there are at most $2\frac{|E|^4}{q^3} \le 2|E|^2q^3$ such quadruples of points $(x,y,z,v)$ such that $x \cdot y = z \cdot y = v \cdot y = t$. For the case when $x=z$, by Theorem 2.2 in \cite{IJM2021}, there are at most $2\frac{|E|^4}{q^3}$ such quadruples of points $(x,y,u,v)$ with $x\cdot y = x\cdot u = u\cdot v =t$.  The conclusion follows.
\end{proof}

\vskip.125in

Let 
$$
A'=\{(x,y,z,v,u)\in A: y\cdot v\neq t, \ x\neq z, \ y\neq u\}.
$$
If $|E|\geq Cq^{\frac{11}{4}}$, then 
$$
|A\setminus A'|\leq 5|E|^2q^3\leq \frac{5}{C^3}\frac{|E|^5}{q^5},
$$
and thus $|A'|\geq \left(\frac{1}{2}-\frac{5}{C^3}\right)\frac{|E|^5}{q^5}$, in particular $|A'|\geq \frac{|E|^5}{4q^5}$.

\vskip.125in

\begin{remark}\label{folding}
We are now ready to take advantage of the symmetry of the configuration in Figure \ref{main_whole_config}.  Ignoring $y_1,y_2,y_3$ for now, we can realize the rest of the configuration by taking a pair of quintuples $(x,y,z,u,v),(x',y,z,u',v)\in A'$ sharing the points $y,z$, and $v$.  
\end{remark}

For ease of notation we let $A'$ denote both the set and its indicator function. Let
$$
f(y,z,v)=\sum_{x,u\in E}A'(x,y,z,u,v).
$$
Then
$$
\frac{|E|^{10}}{16q^{10}}\leq |A'|^2=\left(\sum_{y,z,v\in E}{f(y,z,v)}\right)^2
$$
By Cauchy-Schwarz, and noting that $D_t(y,z)=1$ whenever $f(y,z,v)\neq 0$, this is bounded by
$$
\left(\sum_{y,z,v}{f(y,z,v)^2}\right)\left(\sum_{y,z,v}{D_t(y,z)}\right)
$$
But for $y\in E$, $\sum_z{D_t(y,z)}\leq \frac{5|E|}{q}$, so
$$
\frac{|E|^7}{80q^9}
\leq \sum_{y,z,v}{f(y,z,v)^2}
=\sum_{y,z,v,}\left(\sum_{x,u}{A(x,y,z,u,v)}\right)^2
$$
$$
=\sum_{x,x',y,z,u,u',v}{A(x,y,z,u,v)A(x',y,z,u',v)}.
$$
\vskip.125in
By the one-to-one correspondence noted in Remark \ref{folding}, the number of ordered tuples of vectors $(x_1, x_2, x_3, y_{12}, y_{13}, y_{23}, y_{123}) \in E^7$ such that
\begin{itemize}
	\item $x_1 \cdot y_{123} = x_2 \cdot y_{123} = x_3 \cdot y_{123} = t$
	\item $x_1 \cdot y_{12} = x_2 \cdot y_{12} = t$
	\item $x_1 \cdot y_{13} = x_3 \cdot y_{13} = t$
	\item $x_2 \cdot y_{23} = x_3 \cdot y_{23} = t$
	\item $x_2 \cdot y_{13} \neq t$
	\item $x_1 \neq x_2$, $x_3 \neq x_2$
	\item $y_{123} \neq y_{12}$, $y_{123} \neq y_{23}$
\end{itemize}

\vskip.125in 

is at least $\frac{|E|^7}{80q^9}$. Figure \ref{main_almost_whole_config} below represents such a tuple.

\vskip.125in 

\begin{figure}[h!]
	\includegraphics[scale=0.15]{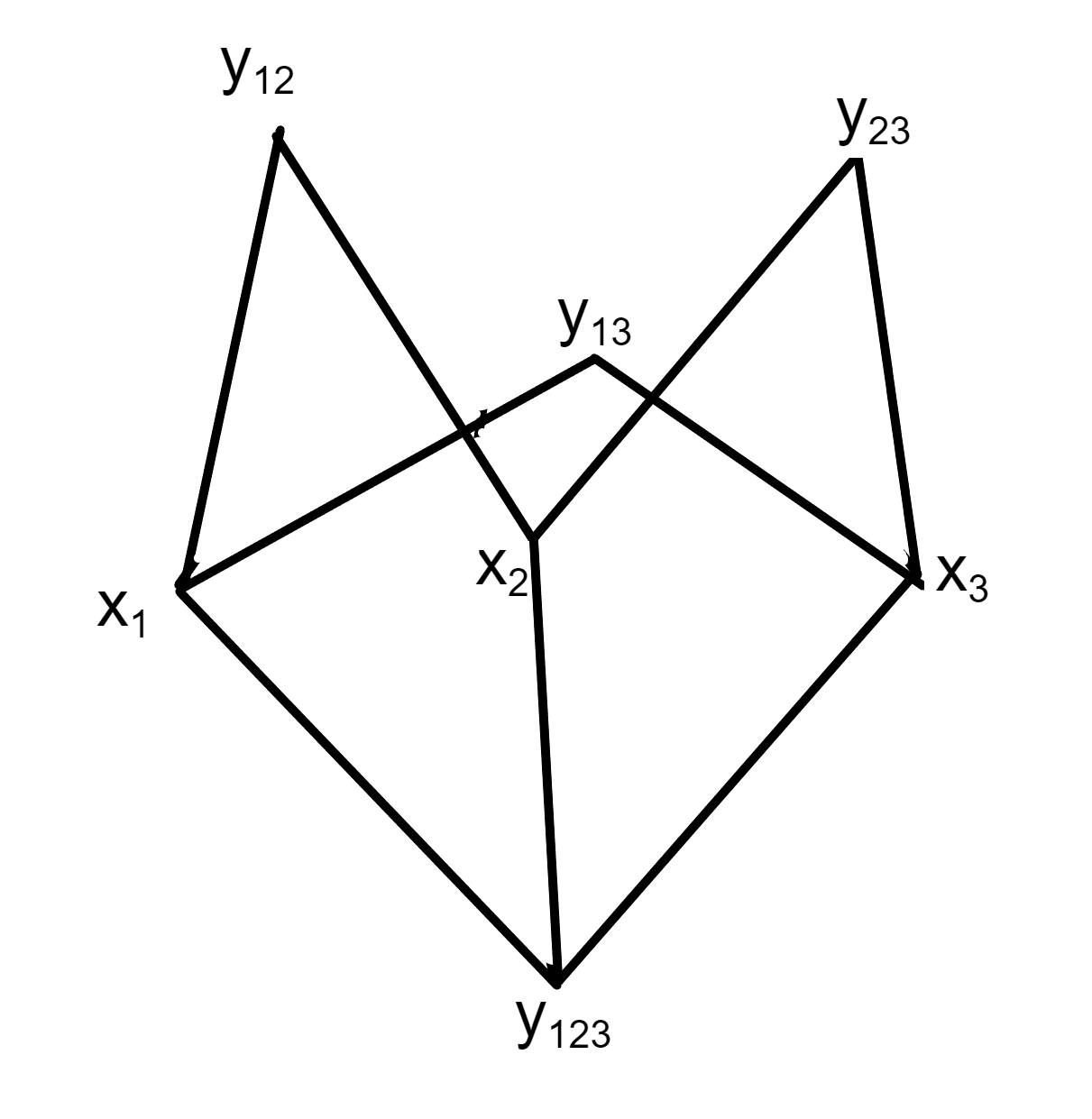}
	\caption{Result of using Cauchy Schwarz}
	\label{main_almost_whole_config}
\end{figure}

We give a lower bound for the number of these tuples where $x_1 \neq x_3$. Suppose $x_1 = x_3$. Then we have six points $x_1, x_2, y_{12}, y_{13}, y_{23}, y_{123} \in E$ where 
\begin{itemize}
	\item $x_1 \cdot y_{123} = x_2 \cdot y_{123} = t$
	\item $x_1 \cdot y_{12} = x_2 \cdot y_{12} = t$
	\item $x_1 \cdot y_{23} = x_2 \cdot y_{23} = t$
	\item $x_1 \cdot y_{13} = t$
	\item $x_2 \cdot y_{13} \neq t$
	\item $x_1 \neq x_2$
	\item $y_{123} \neq y_{12}$, $y_{123} \neq y_{23}$
\end{itemize}

\vskip.125in 

We count the number of such tuples, summing first in $y_{13}$ and then handling the remaining sum with Lemma \ref{degeneracy_case}.  In the notation of Lemma \ref{degeneracy_case}, the sum in the second line of the following calculation corresponds to the case when $y\cdot v =t$.  
$$
	\sum_{x_1, x_2, y_{12}, y_{13}, y_{23}, y_{123} \in E}  D_t(x_1,y_{123})D_t(x_1,y_{12})D_t(x_1,y_{23})D_t(x_2,y_{123})D_t(x_2,y_{12})D_t(x_2,y_{23}) D_t(x_1, y_{13}) 
$$
$$
	\le \frac{5|E|}{q}\sum_{x_1, x_2, y_{12}, y_{23}, y_{123} \in E} D_t(x_1,y_{123})D_t(x_1,y_{12})D_t(x_1,y_{23})D_t(x_2,y_{123})
	D_t(x_2,y_{12})D_t(x_2,y_{23}) 
$$
$$
	\le \frac{5|E|}{q} \cdot |E|^2q^3
	= 5|E|^3q^2
	\leq  \frac{|E|^7}{800q^9}.
$$
It follows that there exist at least $$\frac{|E|^7}{80q^9}-\frac{|E|^7}{800q^9} \ge \frac{9|E|^7}{800q^9}$$ distinct tuples of vectors $(x_1, x_2, x_3, y_{12}, y_{13}, y_{23}, y_{123}) \in E^7$ such that
\begin{itemize}
	\item $x_1 \cdot y_{123} = x_2 \cdot y_{123} = x_3 \cdot y_{123} = t$
	\item $x_1 \cdot y_{12} = x_2 \cdot y_{12} = t$
	\item $x_1 \cdot y_{13} = x_3 \cdot y_{13} = t$
	\item $x_2 \cdot y_{23} = x_3 \cdot y_{23} = t$
	\item $x_2 \cdot y_{13} \neq t$
	\item $x_1 \neq x_2$, $x_3 \neq x_2$, $x_1 \neq x_3$
	\item $y_{123} \neq y_{12}$, $y_{123} \neq y_{23}$
\end{itemize}

\vskip.125in 

Furthermore, for any such tuple, $y_{12} \cdot x_3 \neq t$. To see why, suppose otherwise. Then, both $y_{12}$ and $y_{123}$ lie on the intersection of the planes defined by $x_1 \cdot y = t$, $x_2 \cdot y = t$, and $x_3 \cdot y = t$. The intersection of two of these planes is either a line or the null set, since they are distinct. So, the intersection of all three is either a line, point, or the null set. Since two distinct points lie on the intersection, it must be a line. Furthermore, it must be the same line that is the intersection of any two of these planes. That is, if $y_{13} \cdot x_1 = t$ and $y_{13} \cdot x_3 = t$, then $y_{13} \cdot x_2 = t$ as well, a contradiction. By analogous reasoning $y_{23} \cdot x_1 \neq t$.

\vskip.125in 

Now, fix one such tuple and observe that there are at least $\frac{|E|}{5q}$ vectors $y_1 \in E$ such that $x_1 \cdot y_1 = t$. However, there are at most $q$ such $y_1$ where $x_2 \cdot y_1 = t$, since the intersection of the planes corresponding to $x_1$ and $x_2$ is at most a line. Likewise, there are at most $q$ such $y_1$ where $x_3 \cdot y_1 = t$. Since $\frac{|E|}{5q} > 2q$, there exist a $y_1$ with $x_1 \cdot y_1 = t$, $x_2 \cdot y_1 \neq t$, and $x_3 \cdot y_1 \neq t$. We can also produce $y_2$ and $y_3$ in $E$ with analogous properties. Since there are at most $3y^2$ vectors $y \in E$ such that $y \cdot x_i = t$ for some $i = 1,2,3$, we can also obtain a $y^* \in E$ where $y^* \cdot x_i \neq t$ for all $i = 1,2,3$.

We have obtained a sequence of vectors in $E$, $
\{x_1, x_2, x_3, y_1, y_2, y_3, y_{12}, y_{13}, y_{23}, y_{123}, y^*
\}$ such that 
\begin{itemize}
\item $x_1 \cdot y_{123} = x_2 \cdot y_{123} = x_3 \cdot y_{123} = t$\
\item $x_1 \cdot y_{12} = x_2 \cdot y_{12} = t$, $x_3 \cdot y_{12} \neq t$, and similarly for $y_{13}$ and $y_{23}$
\item $x_1 \cdot y_1 = t$, $x_2 \cdot y_1, x_3 \cdot y_1 \neq t$, and similarly for $y_1$ and $y_2$
\item $x_1 \cdot y^*, x_2 \cdot y^*, x_3 \cdot y^* \neq t$,
\end{itemize}
as desired.

\vskip.125in 

\end{document}